\documentclass[11pt]{amsart}
\usepackage{amssymb, adjustbox, amsbsy}
\usepackage{array, longtable, booktabs, tabularx}
\usepackage{geometry}

\newcolumntype{Y}{>{\centering\arraybackslash}X}

\usepackage{amsfonts, amscd}
\numberwithin{equation}{section}

\usepackage{bm}
\usepackage{verbatim}
\usepackage{graphicx}
\graphicspath{{images/}}
\usepackage{tikz, tikz-cd}
\usepackage{subcaption}
\usepackage{listings}
\usepackage{subfiles}
\usepackage{mathtools}
\usepackage{comment}
\usepackage{enumitem}
\usepackage[all]{xy}
\usepackage{hyperref}
\hypersetup{
    colorlinks=true,
    citecolor=red,
    linkcolor=blue,
    filecolor=magenta,      
    urlcolor=red,
}

\usepackage[textsize=footnotesize,bordercolor=red,linecolor=red,backgroundcolor=red!15]{todonotes}

\newcommand{\Spec}{\mathrm{Spec}}

\newcommand{\Ext}{\operatorname{Ext}}

\newcommand{\Ff}{\mathcal{F}}
\newcommand{\Gg}{\mathcal{G}}

\newcommand{\Ee}{\mathcal{E}}

\newcommand{\Ll}{\mathcal{L}}

\newtheorem{thm}{Theorem}[section]

\newtheorem{lem}[thm]{Lemma}

\newtheorem{consthm}[thm]{Construction-Theorem}

\theoremstyle{definition}

\newtheorem{ques}[thm]{Question}
\newtheorem{rem}[thm]{Remark}

\tikzstyle{wbullet}=[circle, draw=black, fill=white, thick, inner sep=2pt, minimum size=1.5mm]
\tikzstyle{bbullet}=[circle, draw=black, fill=black, inner sep=2pt, minimum size=1.5mm]

\begin{document}

\title{On a question of Koll\'ar and Kov\'acs}
\author{Jihao Liu}
\address{Department of Mathematics, Peking University, No. 5 Yiheyuan Road, Haidian District, Beijing 100871, China}
\address{Beijing International Center for Mathematical Research, Peking University, No. 5 Yiheyuan Road, Haidian District, Beijing 100871, China}
\email{liujihao@math.pku.edu.cn}

\date{\today}

\begin{abstract}
We answer a question of Koll\'ar and Kov\'acs by constructing a flat projective morphism to a smooth curve whose fibers are Cohen--Macaulay and reduced, whose generic fiber is smooth, and for which the first cohomology of the structure sheaf of the fibers is not constant. The main result of this paper is obtained by generative AI, particularly Chatgpt 5.5 pro and the Rethlas system.
\end{abstract}

\maketitle

\section{Introduction}\label{sec:introduction}

We work over the field of complex numbers $\mathbb C$.
In the study of higher direct images of sheaves, Koll\'ar and Kov\'acs proposed the following question.

\begin{ques}[{\cite[Question~14]{KK25}}]\label{ques:kk25}
Is there a flat, projective morphism $g\colon X_C\to C$ to a smooth curve such that
\begin{enumerate}
\item all fibers of $g$ are CM and reduced,
\item the generic fiber of $g$ is smooth, and
\item the function $c\mapsto h^1(X_{C,c},\mathcal{O}_{X_{C,c}})$ is not constant?
\end{enumerate}
\end{ques}

The goal of this paper is to provide a positive answer to Question~\ref{ques:kk25}.

\begin{thm}\label{thm:kk25-14}
There exists a flat projective morphism $g\colon X_C\to C$ satisfying (1)--(3) of Question~\ref{ques:kk25}. More precisely, we have
\begin{equation}\label{eq:kk25-h1-values}
    h^1(X_{C,0},\mathcal{O}_{X_{C,0}})=1\quad \text{and}\quad h^1(X_{C,c},\mathcal{O}_{X_{C,c}})=0\quad \text{for any closed point}\quad c\in C\setminus\{0\}.
\end{equation}
This provides a positive answer to Question~\ref{ques:kk25}.
\end{thm}

We refer the reader to Construction-Theorem~\ref{consthm:B} for the precise construction.

\begin{rem}
It is worth mentioning that, in the construction of Theorem~\ref{thm:kk25-14}, the curve $C$ is obtained from $\mathbb A^1$ by deleting finitely many points away from $0$.
This is used to ensure that the non-special fibers are smooth and to avoid finitely many bad parameters.
It remains interesting to ask whether one can construct an example with $C$ projective.
\end{rem}

\begin{rem}
The main result of this paper is obtained by generative AI by first asking a question to Chatgpt 5.5 pro for general ideas, then putting the idea into the Rethlas system. The Rethlas system then provides the right example.

See \cite{Ju+26} for a detailed introduction to the Rethlas system. Due to the limitation of generative AI, it is possible that we have missed some related references in the literature, and we welcome any comments from experts.
\end{rem}

\subsection*{Acknowledgements}
The author was partially supported by the National Key R\&D Program of China \#\allowbreak 2024YFA1014400. The author would like to thank the Rethlas team, namely Haocheng Ju, Jiedong Jiang, Shurui Liu, Guoxiong Gao, Yuefeng Wang, Zeming Sun, Bin Wu, Liang Xiao, and Bin Dong, for their contributions to the development of Rethlas and its customized version used for the problem studied in this paper. The author would like to thank Kaiyuan Gu, Ruicheng Hu, and Sheng Qin for assistance with the verification of an earlier blueprint of this paper. The author would like to thank Ruochuan Liu and Gang Tian for constant support and encouragement.

\section{Cohomology computation}

In this section, we compute the first cohomology of the structure sheaf of a hypersurface in a projective bundle over $\mathbb P^1$.

\begin{lem}\label{lem:C}
Let $B:=\mathbb P^1$. Let $\Ee$ be a rank $3$ vector bundle over $B=\mathbb P^1$ such that $\det\Ee\cong\mathcal{O}_B$, let $p\colon P:=\mathbb P_B(\Ee)\to B$ be the projective bundle, and let $X\subset P$ be the subscheme defined by a non-zero section of $\mathcal O_P(4)\otimes p^*\mathcal O_B(1)$.
Then
\begin{equation}\label{eq:h1-formula}
H^1(X,\mathcal O_X)\cong H^0(B,\Ee^\vee\otimes\mathcal O_B(-1)).
\end{equation}
In particular:
\begin{enumerate}
    \item If $\Ee\cong \mathcal O_B^{\oplus 3}$, then $h^1(X,\mathcal{O}_X)=0$.
    \item If $\Ee\cong\mathcal O_B(-1)\oplus\mathcal O_B\oplus\mathcal O_B(1)$, then $h^1(X,\mathcal{O}_X)=1$.
\end{enumerate}
\end{lem}
\begin{proof}
The scheme $X$ is a Cartier divisor with defining sequence
\[
0\to \Ff\to\mathcal O_P\to\mathcal O_X\to 0,
\qquad \Ff:=\mathcal O_P(-4)\otimes p^*\mathcal O_B(-1).
\]
The projective bundle $p$ satisfies
$p_*\mathcal O_P=\mathcal O_B$
and $R^i p_*\mathcal O_P=0$ for $i>0$.
Since $B=\mathbb P^1$, we have
$H^1(B,\mathcal O_B)=H^2(B,\mathcal O_B)=0$, and the Leray spectral sequence yields
\[
H^1(P,\mathcal O_P)=H^2(P,\mathcal O_P)=0.
\]
Thus
\[
H^1(X,\mathcal O_X)\cong H^2(P,\Ff).
\]
On every geometric fiber $\mathbb P^2$ of $p$, the restriction of $\Ff$ is
$\mathcal O_{\mathbb P^2}(-4)$ hence has vanishing $H^0$ and $H^1$, so
$$p_*\Ff=R^1 p_*\Ff=0$$ by cohomology and base change (cf.~\cite[III, Corollary~12.9]{Har77}).
Hence
\[
H^2(P,\Ff)\cong H^0(B,R^2 p_*\Ff).
\]
For the projective bundle $p$, relative Serre duality gives
\[
(R^2 p_*\Ff)^\vee
\cong p_*(\Ff^\vee\otimes\omega_{P/B})
= p_*\!\bigl(\mathcal O_P(1)\otimes p^*(\mathcal O_B(1)\otimes\det \Ee)\bigr)
= \Ee\otimes\mathcal O_B(1)\otimes\det \Ee.
\]
Dualizing and using $\det \Ee\cong\mathcal O_B$, we obtain
\[
R^2 p_*\Ff\cong \Ee^\vee\otimes\mathcal O_B(-1).
\]
Therefore,
\[
H^1(X,\mathcal O_X)\cong H^0(B,\Ee^\vee\otimes\mathcal O_B(-1)).
\]
If $\Ee\cong\mathcal O_B^{\oplus 3}$, then
\[
\Ee^\vee\otimes\mathcal O_B(-1)
\cong\mathcal O_B(-1)^{\oplus 3},
\]
so $h^0(B,\Ee^\vee\otimes\mathcal O_B(-1))=0$. If
\[
\Ee\cong\mathcal O_B(-1)\oplus\mathcal O_B\oplus\mathcal O_B(1),
\]
then
\[
\Ee^\vee\otimes\mathcal O_B(-1)
\cong\mathcal O_B\oplus\mathcal O_B(-1)\oplus\mathcal O_B(-2),
\]
so $h^0(B,\Ee^\vee\otimes\mathcal O_B(-1))=1$. The lemma follows.
\end{proof}

\section{Extension family and special section}\label{sec:extension-family}

In this section, we construct the family of vector bundles and the special section used in the construction of the example.

\begin{lem}\label{lem:A}
Let $B=\mathbb P^1$, $T=\mathbb A^1=\Spec\mathbb C[t]$, and let $\pi_B\colon B\times T\to B$ and $\pi_T\colon B\times T\to T$ be the two projections. There exists a rank-two locally free sheaf $\Gg$ on $B\times T$ fitting into
\begin{equation}\label{eq:jumping-extension}
0\to\pi_B^*\mathcal{O}_B(-1)\to \Gg\to\pi_B^*\mathcal{O}_B(1)\to 0
\end{equation}
such that
\begin{equation}\label{eq:G-fibers}
    \Gg_0:=\Gg|_{B\times\{0\}}\cong \mathcal{O}_B(-1)\oplus\mathcal{O}_B(1)\qquad \text{and}\qquad \Gg_c:=\Gg|_{B\times\{c\}}\cong \mathcal{O}_B^{\oplus 2}
\end{equation}
for any $c\not=0$. In particular, if
\begin{equation}\label{eq:E-family}
   \Ee:=\Gg\oplus\mathcal{O}_{B\times T}, 
\end{equation}
then $\det(\Ee_c)\cong\mathcal O_B$ for every closed $c\in T$, where $\Ee_c:=\Ee|_{B\times\{c\}}$.
\end{lem}
\begin{proof}
By the projection formula and the affineness of $T$,
\[
\Ext^1_{B\times T}\bigl(\pi_B^*\mathcal O_B(1),\pi_B^*\mathcal{O}_B(-1)\bigr)
=H^1(B,\mathcal O_B(-2))\otimes_{\mathbb C}\mathbb C[t].
\]
$H^1(B,\mathcal O_B(-2))$ is of dimension $1$. Thus we may pick a generator $e$ of $H^1(B,\mathcal O_B(-2))$ and let $\Gg$ be the extension with class $te$.

$\Gg$ satisfies our requirements. When $c=0$, $ce=0$, so
$$\Gg_0\cong\mathcal O_B(-1)\oplus\mathcal O_B(1).$$
When $c\not=0$, $ce\not=0$ in the one-dimensional group $\Ext^1_{B}\!\bigl(\mathcal O_B(1),\mathcal O_B(-1)\bigr)$, hence proportional to the class of the Euler sequence on $\mathbb P^1$. Thus
$\Gg_c\cong\mathcal O_B^{\oplus 2}$. The in particular part is clear.
\end{proof}

\begin{lem}\label{lem:cons-sigma}
Notation and conditions as in Lemma~\ref{lem:A}. Let $P_0:=\mathbb P_B(\Ee_0)$, let $p_0\colon P_0\to B$ be the associated contraction, and let $y_0\in H^0(P_0,\mathcal{O}_{P_0}(1))$ be induced by the summand $\mathcal{O}_B\subset\mathcal{E}_0$. Pick distinct sections $a,b,s\in H^0(B,\mathcal{O}_B(1))$ and denote their zero loci by the same letters. Then:
\begin{enumerate}
    \item There exists $q\in H^0(B\times T,\Gg)$ lifting $\pi_B^*s$ through the surjection
    $$\Gg\twoheadrightarrow\pi_B^*\mathcal{O}_B(1).$$
    In particular, we may identify $q_0$ as a section in $\mathcal{O}_{P_0}(1)$.
    \item We define
    \begin{equation}\label{eq:sigma-zero}
\sigma_0=ay_0^4+bq_0^4\in
H^0\!\bigl(P_0,\mathcal O_{P_0}(4)\otimes p_0^*\mathcal O_B(1)\bigr).
\end{equation}
Then:
\begin{enumerate}[label=(\roman*)]
    \item $\sigma_0$ is non-zero along any geometric fiber of $p_0$.
    \item Set $K:=\mathbb C(B)$ with fiber coordinates $u,v,w$ under the identification of the splitting
$$\Ee_0\cong\mathcal{O}_B(-1)\oplus\mathcal{O}_B\oplus \mathcal{O}_B(1),$$
then the generic fiber equation of $\sigma_0=0$ has the form $r_vv^4+r_ww^4$ with $r_v,r_w\in K^*$, and this polynomial is square-free in $K[u,v,w]$.
\end{enumerate}
\end{enumerate}
\end{lem}
\begin{proof}
The obstruction to lifting $s$ through $\Gg\twoheadrightarrow\mathcal O_B(1)$ lies in
\[
H^1(B\times T,\pi_B^*\mathcal O_B(-1))
=H^1(B,\mathcal O_B(-1))\otimes_{\mathbb C}\mathbb C[t]=0.
\]
Thus there exists a lift
$q\in H^0(B\times T,\Gg)$ of $s$. This implies (1). On $B\times\{0\}$, the splitting writes 
$$q_0=(q_0',s),\quad  q_0'\in H^0(B,\mathcal O_B(-1))=0,$$ 
so $q_0=(0,s)$. Fix $\beta\in B$ and let $u,v,w$ be the fiber coordinates adapted to the splitting of $\Ee_0$ at $\beta$. Then $y_0|_{p_0^{-1}(\beta)}=v$ and $q_0|_{p_0^{-1}(\beta)}=s(\beta)w$. So
\begin{equation}
\sigma_0|_{p_0^{-1}(\beta)}=a(\beta)v^4+b(\beta)s(\beta)^4w^4.
\end{equation}
If $a(\beta)\not=0$, then the coefficient of $v^4$ is not $0$. If $a(\beta)=0$, then $\beta=a$, so $b(\beta)\not=0$ and $s(\beta)\not=0$, so the coefficient of $w^4$ is not $0$. This proves (2)(i). 

Over $K=\mathbb C(B)$, $q_0$ is a nonzero scalar multiple of the coordinate $w$ adapted to the $\mathcal O_B(1)$-summand.
The generic equation is $r_vv^4+r_ww^4$ with $r_v=a_K$ and $r_w=b_Ks_K^4\in K^*$.
Its partial derivatives are $4r_vv^3$ and $4r_ww^3$.
In characteristic zero, any common factor of $r_vv^4+r_ww^4$, $4r_vv^3$, and $4r_ww^3$ divides
$\gcd(v^3,w^3)=1$.
Hence the polynomial is square-free.
\end{proof}

\section{Construction of the example}\label{sec:divisor-geometry}

In this section, we construct the flat projective family that gives the positive answer to Question~\ref{ques:kk25}.

\begin{consthm}\label{consthm:B}
Notation and conditions as in Lemmas~\ref{lem:A} and~\ref{lem:cons-sigma}. Let
\begin{equation}\label{eq:rho-and-L}
    \rho\colon P:=\mathbb P_{B\times T}(\Ee)\to B\times T,\quad \Ll:=\mathcal{O}_{P}(4)\otimes\rho^*\pi_B^*\mathcal{O}_B(1),
\end{equation}
and let $y\in H^0(P,\mathcal{O}_P(1))$ be the section induced by the $\mathcal{O}_{B\times T}$-summand of $\Ee$. Then there exists $\tau\in H^0(P,\Ll)$ such that
\begin{equation}\label{eq:sigma-total}
\sigma=ay^4+bq^4+t\tau\in H^0(P,\Ll)
\end{equation}
satisfying the following.
\begin{enumerate}
    \item $\sigma|_{P_0}=\sigma_0$.
    \item There exists a finite set $Z\subset T^{*}:=T\setminus\{0\}$ with $C:=T\setminus Z$ satisfying the following. Let $P_C$ be the base change of $P$ over $C$ and let $X_C=(\sigma=0)\subset P_C$ be the scheme defined as a zero divisor in $P_C$. Let $g\colon X_C\to C$ be the induced morphism. Then:
    \begin{enumerate}[label=(\roman*)]
        \item Every fiber over $C\setminus\{0\}$ is smooth.
        \item $g$ is projective and flat.
        \item Every fiber of $g$ is CM and reduced.
        \item We have
        \begin{equation}\label{eq:consthm-h1-values}
    h^1(X_{C,0},\mathcal{O}_{X_{C,0}})=1\quad \text{and}\quad h^1(X_{C,c},\mathcal{O}_{X_{C,c}})=0\quad \text{for any closed point}\quad c\in C\setminus\{0\}.
\end{equation}
    \end{enumerate}
\end{enumerate}
\end{consthm}
\begin{proof}
Let $\eta$ be the generic point of $T$. By Lemma~\ref{lem:A}, $\Ee_\eta\cong\mathcal O_{\mathbb P^1_{\mathbb C(t)}}^{\oplus 3}$, so
$$P_{\eta}\cong\mathbb P_{\mathbb C(t)}^1\times_{\mathbb C(t)}\mathbb P_{\mathbb C(t)}^2\quad \text{and}\quad \Ll_{\eta}\cong\mathcal{O}(1,4).$$
Since $\mathcal{O}(1,4)$ is very ample, with coordinates $[z_0:z_1]$ on $\mathbb P^1$ and $[x_0:x_1:x_2]$ on $\mathbb P^2$, we set
$$F:=z_0(x_0^4+x_1^4)+z_1(x_1^4+x_2^4).$$
Then $F$ cuts out a smooth divisor. By Bertini's theorem, there exists a non-empty Zariski open subset $U$ of $H^0(P_{\eta},\Ll_{\eta})$ of smooth divisors. Since
$$\rho_*\mathcal{O}_P(4)=\mathrm{Sym}^4\Ee,\quad R^i\rho_*\mathcal{O}_P(4)=0\quad \text{for any}\quad i>0,$$
the base change map
$$H^0(P,\Ll)\otimes_{\mathbb C[t]}\mathbb C(t)\xrightarrow{\cong}H^0(P_{\eta},\Ll_{\eta})$$
is an isomorphism.

Since $t\in\mathbb C(t)^*$, the affine map
$$v\mapsto (ay^4+bq^4)|_{P_{\eta}}+tv$$
is a surjection onto $H^0(P_{\eta},\Ll_{\eta})$. Let $V$ be the inverse image of $U$, then $V$ is a non-empty Zariski open subset of $H^0(P,\Ll)\otimes_{\mathbb C[t]}\mathbb C(t)$. 

Choose $m_1,\dots,m_N\in H^0(P,\Ll)$ such that the images of $m_i$ are a $\mathbb C(t)$-basis of $H^0(P,\Ll)\otimes_{\mathbb C[t]}\mathbb C(t)$. The complement of $V$ is cut out by finitely many non-zero polynomial conditions over $\mathbb C(t)$. Since $\mathbb C$ is infinite, by clearing the denominators in $\mathbb C[t]$, we may choose $\lambda_1,\dots,\lambda_N\in\mathbb C$ such that
$$\tau=\sum_{i=1}^N\lambda_im_i\in V.$$
Thus $\sigma|_{P_0}=\sigma_0$ and $\sigma|_{P_{\eta}}$ is a smooth divisor of $P_{\eta}$. In particular, this implies (1).

By openness of smoothness, there exists a proper closed subset $Z$ of $T\setminus\{0\}$ such that $X\to T$ is smooth over $T\setminus(Z\cup\{0\})$. We let $C:=T\setminus Z$ and let $X_C\to C$ be the base change of $X\to T$. This construction implies (2)(i).

We prove (2)(ii). Fix $x\in X_C$ and put $c=(\pi_T\circ\rho)(x)\in C$. Let \(A=\mathcal O_{P_C,x}\) and \(R=\mathcal O_{C,c}\). If \(c\) is the generic point, there is nothing to prove. Otherwise \(R\) is a DVR with uniformizer \(\xi\). Since \(P_C\to C\) is smooth, \(\xi\) is \(A\)-regular. Moreover \(\sigma_c\) is a nonzero section on the integral regular scheme \(P_c\), hence its image is \(A/\xi A\)-regular. Tensoring
\[
0\to A\xrightarrow{\sigma} A\to A/(\sigma)\to0
\]
with \(R/(\xi)\), we get that multiplication by \(\xi\) on \(A/(\sigma)\) is injective. Thus \(A/(\sigma)\) is torsion-free over the DVR \(R\), hence flat. Since $X_C\to P_C$ is a closed immersion and $P_C\to C$ is projective, $X_C\to C$ is projective. This implies (2)(ii).

\medskip

We prove (2)(iii). CM is a local property. For any closed point $c\in C$, $\sigma_c$ cuts out an effective Cartier divisor in the regular threefold $P_c$. Since a Cartier divisor in a regular scheme is CM, $X_c$ is CM for any $c\in C$. For any $c\not=0$ in $C$, $X_c$ is smooth, hence reduced. 

To show that $X_0$ is reduced, by Lemma~\ref{lem:cons-sigma}, we note that $\sigma_0\not=0$ on every $\mathbb P^2$-fiber of $p_0$, so no irreducible component of $X_0$ is the whole fiber. Since \(X_0\) is an effective Cartier divisor in the regular threefold \(P_0\), it is pure two-dimensional. Hence any irreducible component not dominating \(B\) would be a two-dimensional closed subscheme of a fiber \(P_{0,\beta}\cong\mathbb P^2\), hence the whole fiber, which is not possible. Therefore, every irreducible component of $X_0$ dominates $B$. Over $\mathbb C(B)$, the generic fiber of $X_0\to B$ is cut out by $r_vv^4+r_ww^4$, which is square-free by Lemma~\ref{lem:cons-sigma}. Thus $X_0$ is generically reduced at every irreducible component of $X_0$. Since $X_0$ is CM, it is $S_1$, so $X_0$ is reduced.

\medskip

We prove (2)(iv). For each closed $c\in C$, the fiber $X_c\subset\mathbb P_B(\Ee_c)$ is the zero scheme of a nonzero section of
$\mathcal O(4)\otimes p_c^*\mathcal O_B(1)$, and
$\det(\Ee_c)\cong\mathcal O_B$ by Lemma~\ref{lem:A}. We have
\begin{equation}
  \Ee_c\cong\mathcal{O}_B^{\oplus 3}\quad \text{for any}\quad c\not=0\quad  \text{and}\quad \Ee_0\cong \mathcal{O}_B(-1)\oplus\mathcal{O}_B\oplus\mathcal{O}_B(1). 
\end{equation}
(2)(iv) follows from Lemma~\ref{lem:C}.
\end{proof}

\end{document}